\documentclass{amsart}
\usepackage{graphicx,amssymb}
\usepackage{amsmath,amssymb,amscd}
\input xy
\xyoption{all}
\usepackage[all]{xy}
\usepackage{hyperref}

\newcommand{\ra}{\rightarrow}

\newcommand{\PP}{\mathbb P}

\newcommand{\cO}{\mathcal{O}}

\newcommand{\IM}{\mbox{im}}

\newcommand{\Hom}{\operatorname{Hom}}

\newcommand{\coker}{\mbox{coker}}
\newcommand{\rk}{\operatorname{rk}}
\newcommand{\Cl}{\operatorname{Cliff}}
\theoremstyle{plain}
\newtheorem{theorem}{Theorem}[section]
\newtheorem{lem}[theorem]{Lemma}
\newtheorem{prop}[theorem]{Proposition}
\newtheorem{cor}[theorem]{Corollary}

\newtheorem{rem}[theorem]{Remark}

\numberwithin{equation}{section}

\begin{document}
\title[Bundles of rank 3]{Bundles of rank 3 on curves of Clifford index 3}

\author{H. Lange}
\author{P. E. Newstead}

\address{H. Lange\\Department Mathematik\\
              Universit\"at Erlangen-N\"urnberg\\
              Cauerstrasse 11\\
              D-$91058$ Erlangen\\
              Germany}
              \email{lange@mi.uni-erlangen.de}
\address{P.E. Newstead\\Department of Mathematical Sciences\\
              University of Liverpool\\
              Peach Street, Liverpool L69 7ZL, UK}
\email{newstead@liv.ac.uk}

\thanks{Both authors are members of the research group VBAC (Vector Bundles on Algebraic Curves). The second author 
would like to thank the Department Mathematik der Universit\"at 
         Erlangen-N\"urnberg for its hospitality}
\date{\today}
\begin{abstract}
Two definitions of the Clifford index for vector bundles on a smooth projective curve $C$ of genus $g\ge4$ were introduced in a previous paper by the authors.
In another paper the authors obtained results on one of these indices for bundles of rank 3. Here we extend these results in the case where $C$ has classical
Clifford index 3. In particular we prove Mercat's conjecture for bundles of rank 3 for $g \leq 8$ and $g \geq 13$ when $C$ has classical Clifford index 3. 
We obtain complete results in the case of genus 7.
\end{abstract}
\maketitle

\section{Introduction}

Let $C$ be a smooth irreducible projective curve of genus $g\ge4$ defined over an algebraically closed field of characteristic $0$. In previous papers we have defined the Clifford index $\Cl_n(C)$ for semistable 
vector bundles of rank $n$ on $C$. Some years ago Mercat conjectured in \cite{m} that this Clifford index is equal to the classical Clifford index $\Cl(C)$.
Recently counter-examples for this conjecture in ranks 2 and 3 have been discovered (see \cite{lmn}, \cite{fo1}, \cite{fo2}, \cite{ln4}, \cite{ln3}, \cite{cl1}, \cite{ln2}). 
On the other hand, we have $\Cl_n(C) = \Cl(C)$ for all $n$ if $\Cl(C) \leq 2$ (see \cite{m}) and $\Cl_2(C) = \Cl(C)$ if $\Cl(C) \leq 4$ (see \cite[Proposition 3.8]{ln}). 
For $n = 3$ the known counter-examples have $\Cl(C) \geq 4$. So the case $\Cl(C) = 3$ remains open.

In \cite{ln2} we proved that, when $\Cl(C) = 3$, then $\Cl_3(C) \geq \frac{8}{3}$, and also obtained a restricted list of cases in which the value $\frac{8}{3}$ was possible. 
In the present paper we eliminate many of these possibilities, in particular proving Mercat's conjecture in this case for $g \leq 8$ (Theorems \ref{thm4.6} and \ref{thm8}) 
and $g \geq 13$ (\cite[Theorem 6.8]{ln2}, Proposition \ref{prop3.3} and Theorem \ref{thm5.5}). For $g=7$ this was proved by Mercat in some unpublished notes. 

We use the methods and results of \cite{ln} and two ideas of Mercat, namely the dual span construction (see \eqref{eq2.6}) and the existence of a non-zero homomorphism $E \ra L^* \otimes K_C$
where $E$ is a rank-3 bundle satisfying certain conditions, $L$ a line bundle of minimal degree with $h^0(L) = 2$ and $K_C$ the canonical bundle (see Corollary \ref{cor2.4}). 
For the cases $g = 13$ and 14 we use also Hartshorne's version of Noether's Theorem for plane curves (see \cite[Theorem 2.1]{h} and Proposition \ref{prop5.1} below).

We investigate also the existence of rank-3 bundles $E$ with $h^0(E) =4$ or 5. These do not contribute to $\Cl_3(C)$, but they do contribute to another Clifford index $\gamma_3(C)$
as defined in \cite{ln}. For $h^0(E) = 4$ we have a good existence result (Propositions \ref{prop6.3} and \ref{p6.5}). For $h^0(E) = 5$ we have in general only necessary conditions for existence 
(Propositions \ref{prop7.8} and \ref{prop7.9}). For $g = 7$ we have a complete solution (Theorem \ref{thm8.6}).

Section 2 contains background material and is followed in Section 3 by some preliminary results which should have more general applications. In Sections 4 and 5 we discuss 
rank-3 bundles of degree 14 and 16. This is sufficient to show that $\Cl_3(C) = 3$ for any curve of genus 7, 8 or 15. In Section 6 we deal with the case of curves of genus 13 and 14. 
Section 7 is concerned with bundles with $h^0 = 4$ or $h^0 = 5$. In Section 8 we discuss bundles of degree 12 with $h^0 = 5$ and show that for a curve of Clifford index 3 the 
dual span construction gives a bijection between the set of stable rank-$3$ bundles of degree $12$ with $h^0 = 5$ and the set of stable 
rank-$2$ bundles of degree $12$ with $h^0 = 5$. Using this we obtain our result for genus 7.

We suppose throughout that $C$ is a smooth irreducible projective curve of genus $g$. We write $K_C$ for the canonical bundle on $C$. For any vector bundle $E$ on $C$, we write $d_E$ for 
the degree and $\mu(E)$ for the slope $\frac{d_E}{\rk E}$ of $E$. 

In the case $g=7$, the results of this paper were outlined by Vincent Mercat in some unpublished notes dated February 17, 2000. In particular Theorem 4.6 is due to him and the arguments in section 4 are generalisations of his arguments. Theorem 8.6 is also essentially due to Mercat.

\section{Background}

In this section we recall some definitions, notations and results from earlier papers.

For any vector bundle $E$ of rank $n$ on $C$, we define
$$
\gamma(E) := \frac{1}{n} \left(d_E - 2(h^0(E) -n)\right) = \mu(E) -2\frac{h^0(E)}{n} + 2.
$$
If $g \geq 4$, we then define, for any positive integer $n$,
$$
\Cl_n(C):= \min_{E} \left\{ \gamma(E) \;\left| 
\begin{array}{c} E \;\mbox{semistable of rank}\; n, \\
h^0(E) \geq 2n,\; \mu(E) \leq g-1
\end{array} \right. \right\}.
$$
Note that $\Cl_1(C) = \Cl(C)$ is the usual Clifford index of the curve $C$. 
We say that $E$ {\it contributes to} $\Cl_n(C)$ if $E$ is semistable of rank $n$ with $h^0(E) \geq 2n$ and $\mu(E) \leq g-1$. 
If in addition $\gamma(E) = \Cl_n(C)$, we say that $E$ {\it computes} $\Cl_n(C)$. 

The {\em gonality sequence} 
$d_1,d_2,\ldots,d_r,\ldots$ of $C$ is defined by 
$$
d_r := \min \{ d_L \;|\; L \; \mbox{a line bundle on} \; C \; \mbox{with} \; h^0(L) \geq r +1\}.
$$

We shall use the following facts without further reference.
\begin{itemize}

\item $\Cl(C)\le\left[\frac{g-1}2\right]$;

\item $d_1=\Cl(C)+2$ or $\Cl(C)+3$ (\cite[Section 2]{cm});

\item $\min\left\{\Cl(C)+2r,g+r-1\right\}\le d_r\le g+r-\left[\frac{g}{r+1}\right]$ (\cite[Remark 4.4(c) and Lemma 4.6]{ln});

\item if $\Cl(C)\le4$, then $\Cl(C)=\Cl_2(C)$ (\cite{m} and \cite[Proposition 3.8]{ln});

\item if $g\ge7$ and $\Cl_2(C)=\Cl(C)\ge2$, then $$\Cl_3(C)\ge\min\left\{\frac{d_9}3-2,\frac{2\Cl(C)+2}3\right\}$$
(\cite[Theorem 4.1]{ln2}).

\end{itemize}

\begin{lem}\label{lema}
\begin{itemize}
\item[(i)] The smallest degree for which there exist semistable bundles of rank $2$ with $h^0\ge3$ is  $d_2$;
\item[(ii)] if $d_2\ne 2d_1$, then all semistable bundles of rank $2$ and degree $d_2$ have $h^0\le3$ and there exists such a bundle with $h^0=3$;
\item[(iii)] if $d<\min\left\{d_3,\frac{3d_2}2\right\}$, then any semistable bundle of rank $3$ and degree $d$ has $h^0\le3$;
\item[(iv)] if $\frac{d_3}3\le\frac{d_2}2$ and $d_3\ne 3d_1$, then all semistable bundles of rank $3$ and degree $d_3$ have $h^0\le4$ and there exists such a bundle with $h^0=4$;
\item[(v)] if $d_3=3d_1$, then there exists a semistable bundle of degree $d_3$ with $h^0=6$.
\end{itemize} \end{lem}
\begin{proof} These results are all included in \cite[Proposition 4.11 and Theorem 4.15]{ln}.
\end{proof}

The following lemma is due to Paranjape and Ramanan (see \cite[Lemma 3.9]{pr} or \cite[Lemma 4.8]{ln}).

\begin{lem} \label{lemPR}
Let $E$ be a vector bundle of rank $n$ with $h^0(E) \geq n+s,\; s\geq 1$. Suppose that $E$ has no proper subbundle $N$ with $h^0(N) > \rk N$. Then $h^0(\det E) \geq ns+1$ and hence $
d_E \geq d_{ns}
$. 
\end{lem}

\section{Preliminaries}

Let $C$ be a smooth irreducible projective curve of genus $g \geq 4$ over an algebraically closed field of characteristic 0.

\begin{lem} \label{lem2.1} {\em (Mercat)}
Let $E$ be a vector bundle of rank $n$ such that $\gamma(E) < \Cl(C)$ and let $L$ be a line bundle on $C$ of degree $d_1$ with $h^0(L) = 2$.
Then either
$$
h^0(L^* \otimes E) > 0 \quad \mbox{or} \quad h^0(E^* \otimes L^* \otimes K_C) > 0.
$$
\end{lem}

\begin{proof}
The line bundle $L$ is globally generated, so we have an exact sequence
$$
0 \ra L^* \ra H^0(L) \otimes \cO_C \ra L \ra 0.
$$
Tensoring with $E$, taking global sections, applying Riemann-Roch and finally the assumption, we get
\begin{eqnarray*}
2h^0(E) & \leq & h^0(L^* \otimes E) + h^0(L \otimes E)\\ 
& = & h^0(L^* \otimes E) + h^0(E^* \otimes L^* \otimes K_C) + d_E + nd_1 - n(g-1)\\
& < & h^0(L^* \otimes E) + h^0(E^* \otimes L^* \otimes K_C) +\\
&& \hspace{2cm}+ 2(h^0(E) -n) + n \Cl(C) + nd_1 - n(g-1).
\end{eqnarray*}
This gives
$$
h^0(L^* \otimes E) + h^0(E^* \otimes L^* \otimes K_C) > 2n -n \Cl(C) -nd_1 +n(g-1).
$$
Now note that $d_1 = \Cl(C) + i$ with $i = 2$ or 3. 
If $d_1 = \Cl(C) + 2$, then 
$$
h^0(L^* \otimes E) + h^0(E^* \otimes L^* \otimes K_C) > -2n \Cl(C) +n(g-1) \geq 0.
$$
If $d_1 = \Cl(C) + 3$, then
$$
h^0(L^* \otimes E) + h^0(E^* \otimes L^* \otimes K_C) > -2n\Cl(C) +n(g-1) -n.
$$
But in this case $\Cl(C) < \frac{g-2}{2}$. So again $h^0(L^* \otimes E) + h^0(E^* \otimes L^* \otimes K_C)$ is positive.
\end{proof}

\begin{lem} \label{observ}
 Suppose $E$ is a semistable bundle of rank $n$ such that
$$
2(n-1) + n\Cl_n(C) \leq d_E \leq 2ng - 2(2n-1) - n\Cl_n(C).
$$
Then
$$
\gamma(E) \geq \Cl_n(C).
$$
\end{lem}

\begin{proof}
If $h^0(E) \geq 2n$ and $h^1(E) \geq 2n$, this follows from the definition of $\Cl_n(C)$. If $h^0(E) < 2n$, we have
\begin{eqnarray*}
n \gamma(E) &=& d_E -2(h^0(E) -n)\\
& \geq & d_E - 2(n-1) \geq n \Cl_n(C)  
\end{eqnarray*}
by hypothesis. If $h^1(E) < 2n$, we use the fact that $\gamma(K_C \otimes E^*) = \gamma(E)$ and the second inequality in the statement.
\end{proof}

\begin{lem} \label{lem2.2}
Suppose $g \geq 7$ and $2 \leq \Cl_2(C)=\Cl(C) \leq 6$. If $E$ is a bundle contributing to $\Cl_3(C)$ with $\gamma(E) < \Cl(C)$, 
then $E$ has no line subbundle $L$ with 
$d_L \geq \Cl(C)$. 
\end{lem}

\begin{proof}
If $L$ exists, then $\Cl(C) \leq d_L \leq g-1$. By the previous lemma we get $\gamma(L) \geq \Cl(C)$.
It is therefore enough to show that 
\begin{equation} \label{eq2.1}
\gamma(E/L) \geq \Cl(C),
\end{equation}
since then $\gamma(E) \geq \Cl(C)$ contradicting the hypothesis. 

First we claim that
\begin{equation} \label{eqn2.2}
d_{E/L} \leq 4g-6 - 2\Cl(C). 
\end{equation}
In fact,
$$
d_{E/L} = d_E - d_L \leq 3g-3 - \Cl(C) \leq 4g-6 -2\Cl(C),
$$
where the last inequality is valid for $\Cl(C) \leq g-3$ which is always true.

It remains to prove \eqref{eq2.1}.
We have 
$$
d_9 \geq \min \{\Cl(C) + 18, g+8 \}
$$
and hence
$$
\frac{d_9}{3} -2 \geq \frac{2 \Cl(C) + 2}{3}.
$$
It follows that $\gamma(E) \geq \frac{2 \mbox{\tiny Cliff}(C) + 2}{3}$
which implies 
$$
d_E = 3 \gamma(E) + 2(h^0(E) -3) \geq 2 \Cl(C) + 8.
$$
Since $E$ is semistable, we get
\begin{equation} \label{eq2.3}
d_{E/L} \geq \frac{4 \Cl(C) + 16}{3}. 
\end{equation}
This implies that
\begin{equation} \label{eq2.4}
d_{E/L} \geq 2\Cl(C) + 2.
\end{equation}
In fact, if $\Cl(C) \leq 5$, then $4\Cl(C) + 16 \geq 6 \Cl(C) + 6$. If $\Cl(C) = 6$, the inequality \eqref{eq2.3} is $d_{E/L} \geq \frac{40}{3}$, so 
$d_{E/L} \geq 14 = 2 \Cl(C) + 2$.

Now suppose that $E/L$ is semistable. Then \eqref{eqn2.2} and \eqref{eq2.4} imply that the hypotheses of Lemma \ref{observ} hold for $E/L$. So Lemma \ref{observ} gives
$$
\gamma(E/L) \geq \Cl_2(C) = \Cl(C).
$$

If $E/L$ is not semistable, let $M$ be a line subbundle with $d_M > \frac{d_{E/L}}{2} \geq \Cl(C) + 1$. Pulling back to $E$ gives $d_L + d_M \leq \frac{2}{3} d_E \leq 2g-2$.
So
$$
d_M \leq 2g-2-d_L \le 2g -2 - \Cl(C).
$$
So by Lemma \ref{observ}
$$
\gamma(M) \geq \Cl(C).
$$
On the other hand, $d_{(E/L)/M} \geq g-1 \geq \Cl(C)$ by the semistability of $E$ and 
$$
d_{(E/L)/M} < \frac{1}{2}d_{E/L} \leq 2g-3 - \Cl(C)
$$
by \eqref{eqn2.2}. So Lemma \ref{observ} gives $\gamma((E/L)/M) \geq \Cl(C)$. Hence $\gamma(E) \geq \Cl(C)$ contradicting the hypothesis.
\end{proof}

As an immediate consequence of Lemmas \ref{lem2.1} and \ref{lem2.2} we obtain

\begin{cor} \label{cor2.4}
Let the hypotheses be as in Lemma \ref{lem2.2}.  If $L$ is a line bundle of degree $d_1$ with $h^0(L)=2$, then $h^0(E^*\otimes L^*\otimes K_C)>0$.
\end{cor}

An analogue of Lemma \ref{lem2.2} for rank-2 subbundles of $E$ is

\begin{lem} \label{lem2.4}
Suppose $g \geq 7$ and $2 \leq \Cl_2(C) = \Cl(C) \leq 6$. If $E$ is a bundle contributing to $\Cl_3(C)$ with $\gamma(E) < \Cl(C)$ and $d_E \leq 2g$, then $E$ has no rank-$2$ subbundle $F$ with $d_F \geq 2\Cl(C) +2$.
\end{lem}

\begin{proof}
Suppose $F$ exists. If $F$ is not semistable, then $E$ possesses a line subbundle $M$ of degree 
$> 1 + \Cl(C)$. This contradicts Lemma \ref{lem2.2}. So $F$ must be semistable. Moreover
$$
2 + 2\Cl(C) \leq d_F \leq \frac{4g}{3} \leq 4g - 6 -2\Cl(C).
$$
So $\gamma(F) \geq \Cl(C)$ by Lemma \ref{observ}. On the other hand we have
$$
d_{E/F} = d_E - d_F \leq 2g-2 - 2 \Cl(C) < 2g - 2 - \Cl(C)
$$
and $d_{E/F} \geq d_F/2\ge  \Cl(C) +1$.  So by Lemma \ref{observ}, $\gamma(E/F) \geq \Cl(C)$. This contradicts the hypothesis that $\gamma(E)<\Cl(C)$.
\end{proof}

Finally in this section, let $E$ be a generated vector bundle. Define $D(E)$ by the exact sequence 
\begin{equation}\label{eq2.6}
0 \ra D(E)^* \ra H^0(E) \otimes \cO_C \ra E \ra 0.
\end{equation} 
$D(E)$ is called the bundle obtained by the {\it dual span construction} from $E$. Note that $h^0(D(E)^*)=0$ and $D(E)$ is generated. Moreover, if $h^0(E^*)=0$, then $h^0(D(E))\ge h^0(E)$. If, in addition, $h^0(D(E))=h^0(E)$, then $D(D(E))\simeq E$.

For any quotient $G$ of $D(E)$, we have a diagram   
\begin{equation*}
\xymatrix{&& 0 \ar[d] & 0 \ar[d] & 0 \ar[d] &\\
&0 \ar[r] & M \ar[r] \ar[d] & W' \otimes \cO_C \ar^{\alpha}[r] \ar[d] & N \ar[d] \ar@{->>}[r] & \coker (\alpha)  \\
(3.6) &0 \ar[r] & E^* \ar[r] \ar^{\beta}[d] & H^0(E)^* \otimes \cO_C \ar[r] \ar[d] & D(E) \ar[r] \ar[d] & 0 \\
&0 \ar[r] & D^*_{W,G} \ar[r] \ar@{->>}[d] & W \otimes \cO_C \ar[r] \ar[d] & G \ar[r] \ar[d] & 0 \\
&& \coker (\beta)& 0 & 0 &
}
\end{equation*}
with exact rows and columns.
Here the middle row is the dual of \eqref{eq2.6}, $W$ is the subspace of $H^0(G)$ which is the image of $H^0(E)^*$ considered as a subspace of $H^0(D(E))$, $D^*_{W,G}$ 
is defined by the bottom sequence and the vertical sequences are exact. Moreover, $\coker (\alpha) \simeq \coker(\beta)$ by the snake lemma.

\section{Bundles of degree 14 or of degree $>2g$}
Suppose now $\Cl(C) = 3$. Then $g\ge7$ and $d_9\ge15$; it follows that 
$$\Cl_3(C)\ge\frac83,$$
a fact which we assume from now on without further reference. 

Let $E$ be a semistable bundle of rank $3$ with $d_E =14,\; h^0(E) = 6$. Since $\gamma(E)=\frac83$, $E$ computes $\Cl_3(C)$. Since $\Cl_2(C)=3\ge\Cl_3(C)$, it follows from \cite[Theorem 2.4]{ln1} that $E$ is generated. Define $D(E)$ by the exact sequence 
\eqref{eq2.6}.
\begin{lem} \label{lem2.5}
D(E) is stable of rank $3$ and degree $14$. 
\end{lem}

\begin{proof}
The assertions on the rank and degree are obvious. Suppose $D(E)$ is not stable. Then either it has a quotient line bundle of degree
$\leq 4$ which is generated (a contradiction since $h^0(D(E)^*)=0$) or it has a stable quotient rank-2 bundle $G$ of degree $\leq 9$ which is generated.
Since $\Cl_2(C) = \Cl(C) = 3$, this implies that $h^0(G) \leq 3$. Moreover, $G$ cannot be trivial, since $h^0(D(E)^*) = 0$ and so $h^0(G^*) = 0$. Hence $h^0(G) = 3$ and $d_G \geq d_2\ge 7$. 
In the diagram (3.6), we see that $\alpha$ is non-zero since $h^0(E^*)=0$. So $\coker(\alpha)$ and $\coker(\beta)$ are torsion sheaves and $D_{W,G}=D(G)$ injects into $E$. Since $D(G)$ is a line bundle and $d_{D(G)} = d_G \geq 7$, this contradicts Lemma \ref{lem2.2}.
\end{proof}

\begin{prop} \label{prop3.1}
Suppose $g \geq 7, \; n=3$ and $\Cl(C) = 3$. Then there exists no semistable rank-$3$ bundle of degree $14$ with $h^0 = 6$.
\end{prop}

\begin{proof}
Suppose $E$ is such a bundle.
Taking an exterior product of \eqref{eq2.6}, we obtain
\begin{equation} \label{eq3.3}
0 \ra P \ra \wedge^2 H^0(E) \otimes \cO_C \ra \wedge^2 E \ra 0
\end{equation}
with $P$ a vector bundle of rank 12 which fits into an exact sequence 
\begin{equation} \label{eqn3.3}
0 \ra \wedge^2 D(E)^* \ra P \ra D(E)^* \otimes E \ra 0. 
\end{equation} 
From \eqref{eqn3.3} and Lemma \ref{lem2.5} 
we see that $h^0(P)\le1$. With \eqref{eq3.3} this implies that $h^0(\wedge^2 E) \geq 14$. 
But $\wedge^2 E \simeq \det(E) \otimes E^*$ and hence is semistable of degree 28.
So $K_C \otimes \wedge^2 E^*$ is semistable of degree $6g-34$ with $h^0(K_C \otimes \wedge^2 E^*) \geq 3g-17$.

Hence, if $g \geq 8$, then $K_C \otimes \wedge^2 E^*$ is a stable bundle of rank 3 with $h^i \geq 6$ for $i = 1,2$ and $\gamma( K_C \otimes \wedge^2 E^*) = 2$. This contradicts
the fact that $\Cl_3(C)\ge \frac83$.

If $g = 7$ , then $K_C \otimes \wedge^2E^*$ is stable of degree 8 with $h^0(K_C \otimes \wedge^2E^*) \geq 4$. 
This contradicts Lemma \ref{lema}(iii), since $d_2 =7$ and $d_3 =9$. 
\end{proof}

An interesting consequence of Proposition \ref{prop3.1} is

\begin{prop} \label{prop3.3}
Suppose $C$ is a curve of genus $15$ with $\Cl(C) =3$. Then $\Cl_3(C) = 3$. 
\end{prop}

\begin{proof}
If $\Cl_3(C)=\frac83$, then, by \cite[Proposition 6.7]{ln2}, we have $d_2 = 7$. This implies that $C$ is a smooth plane septic. Then, by \cite[Theorem 5.6]{ln2}, any bundle $E$ contributing to $\Cl_3(C)$
with $\gamma(E) < 3$ has degree 14. By Proposition \ref{prop3.1}, such a bundle does not exist. 
\end{proof}

\begin{prop}  \label{prop3.4}
Suppose $g \geq 7, \; n=3$ and $\Cl(C) = 3$. Then there exists no semistable bundle $E$ of degree $> 2g$ contributing to $\Cl_3(C)$ with $\gamma(E) < 3$.
\end{prop}

\begin{proof}
Suppose $E$ is a bundle with these properties; in particular $\gamma(E) = \frac{8}{3} < \Cl(C)$. By \cite[Propositions 6.5, 6.6 and 6.7]{ln2}, 
the only possibilities are 
\begin{itemize}
\item $g=7, \; d_E = 16,  \; h^0(E) = 7$;
\item $g=7 \; \mbox{or} \; 8, \; d_E = 18,  \; h^0(E) = 8$;
\item $g=8, \; d_E = 20,  \; h^0(E) = 9$.
\end{itemize}
Since there is no curve of Clifford dimension $>1$ of genus 7 or 8, we must have $d_1 = 5$.

Suppose $L$ is a line bundle computing $d_1$. 
By Corollary \ref{cor2.4}, there is a non-zero homomorphism 
$$
\varphi: E \ra L^* \otimes K_C.
$$ 
Denote by $F$ its kernel.
Note that 
$$
\deg (L^* \otimes K_C) = 2g-7 \quad  \mbox{and} \quad  h^0(L^* \otimes K_C) = g-4.
$$

Suppose first that $g=7, \;d_E = 16$ and $h^0(E) = 7$. 
If $F$ is not semistable, then $F$ (and hence $E$) has a line subbundle of degree $\geq 5$ contradicting Lemma \ref{lem2.2}.
So $F$ is semistable. If $\varphi$ is surjective, then $d_F = 9$ and $h^0(F) \geq 4$ contradicting $\Cl_2(C) = 3$.
If $\varphi$ is not surjective, then
$h^0(F) \geq 5$. So $d_F \geq 12$ contradicting the semistability of $E$.

Let $g = 7, \;d_E = 18$ and $h^0(E) = 8$. So $d_F \geq 11, h^0(F) \geq 5$. Lemma \ref{lem2.2} implies that $F$ is semistable. Since $h^0(F) \geq 5$ and 
$\Cl_2(C) = 3$, this implies $d_F \geq 12$. So $\varphi$ is not surjective. Hence $h^0(F) \geq 6$ implying $d_F \geq 14$, which contradicts the semistability of $E$.

Now suppose $g=8, \; d_E = 18$ and $h^0(E) = 8$.Then $L^* \otimes K_C$ has degree 9 and $h^0(L^* \otimes K_C) = 4$. So $d_F \geq 9$ and $h^0(F) \geq 4$. 
If $F$ is not semistable, it has a line subbundle of degree $\geq 5$ which contradicts Lemma \ref{lem2.2}. 
So $F$ is semistable. If $\varphi$ is surjective, then $d_F = 9, \; h^0(F) \geq 4$ contradicting $\Cl_2(C) = 3$. If $\varphi$ is not surjective, we have $h^0(F) \geq 5$ and 
hence $d_F \geq 12$. So by semistability of $E$ we have $d_F = 12, \; d_{E/F} = 6$ which implies $h^0(E/F) \leq 2$. Since $h^0(F) = 5$, this contradicts $h^0(E) = 8$.

Finally, suppose $g = 8, \; d_E = 20$ and $h^0(E) = 9$. So $d_F \geq 11$. If $F$ is not semistable, we get a contradiction to Lemma \ref{lem2.2}. 
Hence $F$ is semistable. If $\varphi$ is surjective, then $d_F = 11$ and $h^0(F) \geq 5$. This is a contradiction since $\Cl_2(C) = 3$. If $\varphi$ is not surjective, $h^0(F) \geq 6$
giving $d_F \geq 14$. This contradicts the semistability of $E$.
\end{proof}

Combining Propositions \ref{prop3.1} and \ref{prop3.4} we get

\begin{theorem} \label{thm4.6}
For a curve $C$ of genus $7$ with $\Cl(C) = 3$, we have
$$
\Cl_3(C) = 3.
$$ 
\end{theorem}

\section{Bundles of degree 16}

Suppose $g\ge8$, $\Cl(C) =3$ and $E$ is semistable of rank 3 and degree 16 with $h^0(E) = 7$. Again $E$ computes $\Cl_3(C)$ and is generated. Define $D(E)$ by the exact sequence \eqref{eq2.6}
and note that now $D(E)$ is of rank 4 with $h^0(D(E)) \geq 7$.

\begin{lem} \label{lem4.1}
$D(E)$ is stable of degree $16$. In fact,\\
{\em (1)} $D(E)$ has no quotient line bundle of degree $\leq 4$;\\
{\em (2)} $D(E)$ has no stable quotient bundle of rank $3$ and degree $\leq 14$;\\ 
{\em (3)} $D(E)$ has no stable quotient bundle of rank $2$ and degree $\leq 9$.
\end{lem}

\begin{proof}
(1) Suppose $D(E)$ has a quotient line bundle of degree $\leq 4$. This cannot be generated, contradicting the fact that $D(E)$ is generated.

(2) Suppose $D(E)$ has a stable quotient bundle $G$ of rank 3 and degree $\leq 14$. If $h^0(G) \geq 6$, we get a contradiction to Proposition \ref{prop3.1} or to $\Cl_3(C) \geq \frac{8}{3}$. 
So $h^0(G) \leq 5$ and moreover $G$ is generated, since $D(E)$ is generated.

$G$ cannot be trivial, since $h^0(G^*) = 0$. So $h^0(G) \geq 4$. Consider the diagram (3.6).  Since $H^0(E^*) =0$, the map $\alpha$ is non-zero. So $\coker(\alpha)$ and $\coker(\beta)$ are torsion sheaves. This implies that $E$ has a subbundle $F$ generated by 
$ D_{W,G}$ of rank $\dim W - 3 = 1$ or 2. We have
$$
d_F \geq d_G  \geq 9,
$$
contradicting Lemma \ref{lem2.2} or \ref{lem2.4}.

(3) Suppose $D(E)$ has a stable quotient bundle $G$ of rank 2 and degree $\leq 9$. Then $G$ is generated with $h^0(G) \geq 3$. If $h^0(G) \geq 4$, this contradicts $\Cl_2(C) = 3$. So $h^0(G) = 3$.

Using diagram (3.6), we note that now $N$ has rank $2$. The map $\alpha$ is again non-zero. If it has rank 2, we  obtain a line subbundle $F$ of $E$ with $d_F\ge d_G \geq d_2 \geq 7$; this contradicts Lemma \ref{lem2.2}. If $\alpha$ has rank $1$, then $\coker(\beta)$ has rank $1$ and $M\simeq E^*$. It follows that $\IM(\alpha)$ generates a subbundle of $D(E)$ of degree at least $16$. The quotient by this subbundle therefore has rank $3$ and degree $\le0$, which is a contradiction since it is generated and $h^0(D(E)^*)=0$.
\end{proof}

\begin{prop} \label{prop4.2}
Suppose $g \geq 8$ and $\Cl(C) = 3$. Then there is no semistable rank-$3$ bundle of degree $16$ with $h^0(E) = 7$.
\end{prop}

\begin{proof} 
Suppose $E$ is such a bundle. Then
$\wedge^2E$ is semistable of degree 32.
 
Suppose there exists a non-zero homomorphism 
$$
\varphi: D(E) \ra E.
$$
Let $F$ denote the subbundle of $E$ generated by $\IM (\varphi)$. If $\rk F = 1$, then $d_F \geq 5$ by Lemma \ref{lem4.1}, which contradicts Lemma \ref{lem2.2}.

Suppose $\rk F = 2$. Then $d_F \geq 9$ by stability of $D(E)$. This contradicts Lemma \ref{lem2.4}.

If $\rk F = 3$, then $F = E$ and $\IM (\varphi)$ is a subsheaf of $E$ of maximal rank. We claim that $d_{{\small \IM} (\varphi)} \geq 15$.
If $\IM (\varphi)$ is stable, this follows from Lemma \ref{lem4.1}(2). If $\IM (\varphi)$ is not stable and
$d_{{\tiny \IM}(\varphi)} \leq 14$, then either it possesses a quotient line bundle of degree $\leq 4$ or a 
stable quotient bundle of rank 2 of degree $\leq 9$. Again the claim follows from Lemma \ref{lem4.1}(1) and (3).

It follows that $\varphi$ cannot drop rank at 2 points. Since the number of conditions for $\varphi$ to drop rank at any chosen point is 2, we conclude that $\dim \Hom (D(E),E) \leq 4$. 
By \eqref{eq3.3} and \eqref{eqn3.3} we obtain
$$
h^0(\wedge^2 E) \geq 17.
$$
So $K_C \otimes \wedge^2 E^*$ is semistable of degree $6g -38$ with $h^0(K_C \otimes \wedge^2 E^*) \geq 3g -18$. This contradicts $\Cl_3(C) \geq \frac{8}{3}$.
\end{proof}

Combining Propositions \ref{prop4.2}, \ref{prop3.1} and \ref{prop3.4}, we obtain

\begin{theorem} \label{thm8}
For a curve $C$ of genus $8$ with $\Cl(C) = 3$, we have
$$
\Cl_3(C) = 3.
$$
\end{theorem}

\section{Genus 13 and 14}

We start with a statement of a part of Hartshorne's version of Noether's Theorem in the case of an irreducible plane curve of degree 7.

\begin{prop} {\em (\cite[Theorem 2.1]{h}) \label{prop5.1}}
Let $\Gamma$ be an irreducible plane curve of degree $7$ and $Z$ a closed subscheme of finite length $\ell \geq 1$. Denote by $L(Z)$ the associated torsion-free sheaf.
Suppose $\ell \leq 14$ and write $\ell = 7r-e$ with $r=1$ or $2$ and $0 \leq e \leq 6$. Then we have
$$
h^0(L(Z)) \leq  \left\{ \begin{array}{lll}
                          \frac{1}{2}r(r+1) & \mbox{if} & e > r+1,\\
                          \frac{1}{2}(r+1)(r+2) -e & \mbox{if} & e \leq r+1.
                         \end{array} \right.
$$
Furthermore, equality occurs if and only if

{\em (a)} $Z = \Gamma \cap \Gamma' + Z_0$ where $\Gamma'$ is a curve of degree $r-1$ and $Z_0$ a subscheme of length $7-e$, in the first case, or 

{\em (b)} $Z = \Gamma \cap \Gamma'' -E$ where $E$ is a subscheme of length $e$ and $\Gamma''$ a curve of degree $r$ containing $E$, in the second case. 
\end{prop}

\begin{lem} \label{lem5.2}
Let $\pi: C \ra \Gamma$ be the normalization of an irreducible plane curve $\Gamma$ of degree $7$ 
such that $C$ has genus $14$ and Clifford index $3$. The only line bundles computing $\Cl(C)$ are the hyperplane bundle $H$ and the pencil of degree $5$ obtained by projection from the 
singular point. 
\end{lem}

\begin{proof}
The genus formula for plane curves implies that $\Gamma$ has a unique singular point which is an ordinary node or cusp. 
Let $L$ be a line bundle computing $\Cl(C)$. By definition, we have  $d_1\le d_L\le g-1$, so
$$
5 \leq d_L \leq 13.
$$
Note that $\pi_*L$ is not locally free and $d_{\pi_*L} = d_L+1$. We apply Proposition \ref{prop5.1} to $\pi_*L$.
Our restrictions on the degree imply $r=1, \; e\leq 1$ or $ r=2$. 

Suppose $r=1$. If $e=0$, we get $d_L = 6$ and $h^0(\pi_*L) \leq 3$. So $\gamma(L) \geq 2$ with equality only if $h^0(\pi_*L) = 3$. By Proposition \ref{prop5.1}(b) this is impossible, 
since $\pi_*L$ is not locally free. So $h^0(\pi_*L) \leq 2$ and $\gamma(L) \geq 4$.
If $e=1$, we get $d_L =5$ and $h^0(\pi_*L) \leq 2$. So $\gamma(L) = 3$ if and only if $h^0(\pi_*L) = 2$. According to Proposition \ref{prop5.1}(b), 
this happens only for the pencil given by projection 
from the singular point.

If $r=2$ and $e > 3$, we have $d_L \geq 7$ and $h^0(\pi_*L) \leq 3$. So $\gamma(L) \geq 3$ with equality if and only if $d_L = 7$ and $h^0(\pi_*L) =3$. 
According to Proposition \ref{prop5.1}(a), this happens
only if $L$ is the hyperplane bundle.
Finally, if $r=2$ and $e \leq 3$, then $d_L = 13 -e$ and $h^0(\pi_*L) \leq 6-e$. So $\gamma(L) \geq 4$, since $e=0$ is not possible.
\end{proof}

\begin{lem} \label{lem5.3}
Let $\pi: C \ra \Gamma$ be the normalization of an irreducible plane curve $\Gamma$ of degree $7$ such that $C$ has genus $13$ and Clifford index $3$.  
The only line bundles computing $\Cl(C)$ are the hyperplane bundle $H$ and one or two pencils of degree $5$ obtained by projection from the singular points.
\end{lem}

\begin{proof}
The curve $\Gamma$ has either 2 singular points, each of which is an ordinary node or cusp, or a tacnode or a second order cusp. Let $L$ be a line bundle computing $\Cl(C)$; now
$$
5 \leq d_L \leq 12.
$$
Note that $\pi_*L$ is not locally free and $d_{\pi_*L} = d_L + 2$. We apply Proposition \ref{prop5.1} to $\pi_*L$.
Our restrictions on the degree imply $r=1, e=0$ or $r=2$.

If $r=1, e=0$, we have $d_L=5$ and $h^0(\pi_*L) \leq 3$. But $\pi_*L$ does not have the form (b) of Proposition \ref{prop5.1}. Hence $h^0(\pi_*L) \leq 2$ 
and $\gamma(L) \geq 3$ with equality only if $h^0(\pi_*L)=2$. The only pencils of degree 5 are given by projection from a singular point.

If $r=2, e>3$, Proposition \ref{prop5.1} gives $h^0(\pi_*L) \leq 3$ and $6 \leq d_L \leq 8$. If $d_L = 6$, then $\pi_*L$ is not of the form (a) of Proposition \ref{prop5.1}.
So in this case $h^0(\pi_*L) \leq 2$ and $\gamma(L) \geq 4$. If $d_L =7$, then $\gamma(L) \geq 3$ with equality only if $h^0(\pi_*L) =3$. It follows 
from (a) of Proposition \ref{prop5.1} that $L \simeq H$. If $d_L = 8$, then $\gamma(L) \geq 4$.  
 
Finally, suppose $r=2, \; e \leq 3$. Then Proposition \ref{prop5.1} gives $h^0(\pi_*L) \leq 6 - e$. Moreover, $d_L = 12 -e$. So $\gamma(L) \geq 2 + e$.
If $e=0$, we note that $\pi_*L$ is not of the form (b) of Proposition \ref{prop5.1}. In this case $h^0(\pi_*L) \leq 5$ and $\gamma(L) \geq 4$. 
Otherwise we can only have $\gamma(L) = 3$ if $e = 1$ and $h^0(\pi_*L) = 5$. So $d_L = 11$ and $\pi_*L$ is not of the form (b) of Proposition \ref{prop5.1}.
\end{proof}
\begin{rem} {\em
It can be proved from Proposition \ref{prop5.1} that, if $\pi:C\to\Gamma$ is the normalization of an irreducible plane curve of degree 7 such that $C$ has genus $13$ or $14$, then $\Cl(C)=3$. In fact, this can be extended to the case when $C$ has genus $\ge9$, provided the only singularities of $\Gamma$ are ordinary nodes or cusps (see \cite[Theorem 2.3]{ck} and \cite[Corollary 2.3.1]{cm}).} 
\end{rem}

\begin{theorem} \label{thm5.5}
For a curve $C$ of genus $13$ or $14$ with $\Cl(C) = 3$, we have 
$$
\Cl_3(C) = 3.
$$ 
\end{theorem}

\begin{proof}
We have seen that, if $\Cl_3(C) < 3$, then $\Cl_3(C) = \frac{8}{3}$ and, by \cite[Proposition 6.7]{ln2}, $d_2 = 7$ and any bundle computing $\Cl_3(C)$ fits into an exact sequence
$$
0 \ra F \ra E \ra M \ra 0
$$
where $\rk F = 2, \; h^0(F) =3,\; d_F = 7$ and $M$ is a line bundle of degree $\geq 7$ such that either $M$ or $K_C \otimes M^*$ computes $\Cl(C)$. Lemmas \ref{lem5.2} 
and \ref{lem5.3} imply that $F \simeq D(H)$ and $M\simeq H, \;K_C \otimes H^*$ or $ K_C \otimes L^*$ with $L$ a pencil of degree 5.

If $M \simeq H$, then $d_E = 14$ which is impossible by Proposition \ref{prop3.1}. If $M \simeq K_C \otimes H^*$ or $K_C \otimes L^*$, then $d_E \geq 2g-2$. These cases can be eliminated since $2g-2 > g + \frac{21}{2}$ (see \cite[Remark 5.10]{ln2}). 
\end{proof}

\begin{rem} {\em
In genus 12 the arguments above allow the possibility of a line bundle of degree 11 computing $\Cl(C)$.
In fact, in this case we have $r=2, e=0$. Since $\pi_*L$ is not locally free, Proposition \ref{prop5.1} gives 
$h^0(\pi_*L) \leq 5$, but (b) does not apply. 
In addition we no longer have $2g-2 > g + \frac{21}{2}$, so we have to allow the possibility that $M\simeq K_C\otimes H^*$ in the proof of Theorem \ref{thm5.5}.}
\end{rem}

\section{Bundles with 4 or 5 sections}

In this section we consider the existence of semistable bundles $E$ of rank 3 with $4 \leq h^0(E) \leq 5$
 on a curve $C$ of Clifford index 3. We recall that in \cite{ln} we defined, for all $n \geq 1$,
$$
\gamma_n(C) = \min \left\{ \gamma(E) \; \left| \begin{array}{c}
                                                                E \; \mbox{semistable of rank}\; n \; \mbox{with}\\
                                                                h^0(E) \geq n+1 \; \mbox{and} \; \mu(E) \leq g-1
                                                               \end{array} \right. \right\}.
$$
Clearly 
$$
\gamma_n(C) \leq \Cl_n(C).
$$

We have investigated the case $n=2$ in \cite{ln3}. For $n=3$, we know from \cite[Theorem 6.1]{ln} that, if $\frac{d_2}{2} \geq \frac{d_3}{3}$, then 
\begin{equation} \label{eq6.1}
\gamma_3(C)= \min \left\{ \Cl_3(C), \frac{1}{3}(d_3 -2) \right\}
\end{equation}

Recall that $d_3 \leq 3d_1$. 

\begin{prop} \label{prop6.1}
Let $C$ be a curve with $\Cl(C) = 3$ and $d_3 = 3d_1$. Then 
$$
\gamma_3(C) = \Cl_3(C) = 3.
$$
Moreover, no bundle with $h^0 = 4$ or $5$ computes $\gamma_3(C)$.
\end{prop}

\begin{proof}
By \eqref{eq6.1}, $\gamma_3(C) =
\Cl_3(C)$.
Any curve $C$ of Clifford index $3$ for which $\Cl_3(C) < \Cl(C)$ has genus $g \leq 12$ by the results of earlier sections. Hence $d_3 \leq 12$. 
Since $d_1 \geq 5$, we have $d_3 < 3d_1$. This proves the first assertion. 

For the second assertion, Lemma \ref{lema}(iii) implies that any bundle $E$ contributing to $\gamma_3(C)$ has degree at least $d_3$.
So, if $h^0(E)\le5$, 
$$
\gamma(E) \geq \frac{1}{3}(d_3 - 2(h^0(E) -3)) \geq \frac{1}{3}(d_3 - 4) > d_1 -2 \geq \Cl_3(C).
$$
\end{proof}

Such curves exist; in fact, by \cite[Remark 4.5(c)]{ln}, the general pentagonal curve of genus $g\ge22$ has the properties of Proposition \ref{prop6.1}.

Suppose from now on that $d_3 < 3d_1$. If also $\frac{d_3}{3} \leq \frac{d_2}{2}$, then, by Lemma \ref{lema} (iv), there exists a semistable bundle $E$ of rank $3$ with
\begin{equation}\label{eq6.2}
d_E=d_3, h^0(E)=4,\mbox{ and hence } \gamma(E)=\frac13(d_3-2).
\end{equation}

\begin{prop} \label{prop6.3} 
Let $C$ be a curve with $\Cl(C) = 3$ such that one of the following holds:
\begin{itemize}
\item $g \leq 10$;
\item $d_2=7$ and $g\le13$;
\item $d_2\ge8$ and $d_3\le11$.
\end{itemize} 
Then the bundles given by \eqref{eq6.2} compute $\gamma_3(C)$. 
\end{prop}

\begin{proof}
We require to prove that
\begin{equation}\label{eq6.3}
\frac{d_3}3\le\frac{d_2}2 \mbox{ and }\frac{1}{3}(d_3-2) \leq \Cl_3(C).
\end{equation}
Since $d_2\ge7$ and $\Cl_3(C)\ge\frac83$, this is clear if $d_3\le10$. This holds if $g\le9$. If $d_2=7$ and $g\le13$, we consider the plane model $\Gamma$ of $C$. This has 
either $2$ or more double points or one double point, which is neither an ordinary node nor a cusp. It follows that there is a $3$-dimensional family of conics passing 
through one or two singular points with total multiplicity at these points $\ge4$. Hence $d_3\le14-4=10$ and the previous argument works. If $d_2\ge8$ and $d_3\le11$, 
then \eqref{eq6.3} still holds. Finally, note that, if $g=10$, then $d_3\le11$, so this is covered by one of the other cases.
\end{proof}

\begin{rem} {\em
If $C$ is a curve of Clifford dimension 3, i.e. a smooth intersection of 2 cubics in $\PP^3$, then
$\Cl_3(C) = \Cl(C) = 3$ and $d_3 = 9$. So again the bundles given by \eqref{eq6.2}
compute $\gamma_3(C)$. }
\end{rem}

In the case $\frac{d_3}{3} > \frac{d_2}{2}$ we have the following proposition which extends \cite[Proposition 3.5]{ln1}.

\begin{prop} \label{p6.5}
If $\frac{d_3}{3} > \frac{d_2}{2}$, then the minimal degree for which there exists a semistable bundle $E$ of rank $3$ 
with $h^0(E) = 4$ is $[\frac{3d_2 + 1}{2}]$. So 
$$
\gamma(E) = \frac{1}{3} \left( \left[ \frac{3d_2 + 1}{2} \right] - 2 \right) \leq \frac{1}{3}(d_3-2).
$$
\end{prop}

\begin{proof}
If $E$ is a semistable bundle of rank $3$ with $d_E < \frac{3d_2}{2}$, then $h^0(E) \leq 3$ by Lemma \ref{lema}(iii). 
The existence of $E$ was proved in \cite[Proposition 3.5]{ln1}. 
\end{proof}

\begin{cor} \label{cor6.6}
If $\Cl(C) = 3$ and $\frac{d_3}{3} > \frac{d_2}{2}$, then $\gamma(E) \geq 3$ with equality if and only if 
$d_2 = 7$.
\end{cor}
\begin{proof}
Since $d_2\ge7$, this follows at once from the proposition.
\end{proof}

\begin{prop} \label{prop6.7}
If $C$ is either a smooth plane septic or the normalization  of a plane septic with one ordinary node or cusp and $E$ is as in Proposition \ref{p6.5}, then $\gamma(E) = 3$.
\end{prop}

\begin{proof}
Clearly $d_2 = 7$ and $\Cl(C) = 3$. Moreover, $g = 14$ or $15$. So, by Theorem \ref{thm5.5},
$\Cl_3(C) = 3$. By Proposition \ref{prop5.1}, we have $d_3 = 11$ or $12$. So Corollary \ref{cor6.6} gives the result.
\end{proof}

\begin{rem}\label{rem6.2}
{\em The only remaining cases are when $d_2\ge8$ and $d_3\ge12$. In these cases, it follows either from Lemma \ref{lema}(3) or Corollary \ref{cor6.6} 
that there are no bundles with $h^0=4$ computing $\gamma_3(C)$.}
\end{rem}

We consider now the case $h^0(E) = 5$. In view of Proposition \ref{prop6.1}, we can assume that $d_3 < 3d_1$.

\begin{prop} \label{prop7.8}
Suppose $\frac{d_3}{3}\leq \frac{d_2}{2}$. Let $E$ be a semistable bundle of rank $3$ with $h^0(E) =5$ such that $\gamma(E) = \gamma_3(C)$. Then there exists a non-trivial extension
\begin{equation} \label{e7.4}
0 \ra F \ra E \ra N \ra 0
\end{equation}
with $F$ computing $\gamma_2(C), \; d_F = d_2$ and $d_N = d_1$. Moreover, one of the following possibilities holds.
\begin{enumerate}
 \item[I.] $d_1 = 5, \;d_2 = 8, \; d_3 = 11$ or $12, \; \gamma(E) = 3$;
\item[II.] $d_1 = 5, \; d_2 = 7,\;  d_3 = 10, \; \gamma(E) = \frac{8}{3}$.
\end{enumerate}
\end{prop}

\begin{proof}
If $E$ has a line subbundle with $h^0 \geq 2$, then, by \cite[Proposition 4.25(c)]{ln},
$$
\gamma(E) > \Cl(C).
$$
If $E$ has no proper subbundle with $h^0 >\rk$, then, by Lemma \ref{lemPR}, $d_E \geq d_6$ and hence
\begin{equation} \label{e7.5}
\gamma(E) \geq \frac{1}{3}(d_6 -4) > \frac{1}{3}(d_3-2). 
\end{equation}
So we may assume that $E$ has no line subbundle with $h^0 \geq 2$, but has a subbundle $F$ of rank 2 with $h^0(F) \geq 3$.

If $h^0(F) \geq 4$, then $d_F \geq d_4$ by Lemma \ref{lemPR}. By semistability of $E$, we have
$d_E \geq \frac{3}{2} d_4 \geq 15$, since $d_4 \geq \min \{ \Cl(C) + 8, g+3 \} \geq 10$. So
$\gamma(E) \geq \frac{11}{3} > 3$.
We can therefore suppose that $h^0(F) = 3$ and write $N = E/F$, so that $h^0(N) \geq 2$. Note first that $d_F \geq d_2$ by Lemma \ref{lemPR} and $d_N \geq d_1$.

If $d_1 = 5$ and $d_2 \geq 8$, then $\Cl_3(C) = \Cl(C) = 3$ and 
$$
\gamma(E) \geq \frac{1}{3}(13 -4) = 3
$$ 
with equality if and only if $d_F = d_2 = 8$ and $d_N = d_1$. So for $\gamma(E) = \gamma_3(C)$ we need also $3 \leq \frac{1}{3}(d_3 -2)$, i.e. $d_3 \geq 11$. 
Also $\frac{d_3}{3} \leq \frac{d_2}{2} = 4$, so $d_3 \leq 12$. This gives case I.

If $d_1 = 5$ and $d_2 = 7$, then $d_E \geq 12$. So
$$
\gamma(E) \geq \frac{1}{3}(12 - 4) = \frac{8}{3}.
$$
For equality we need $d_E = 12$ and hence $d_F = 7$ and $d_N = 5$. For $E$ to compute $\gamma_3(C)$ we need also $\frac{1}{3}(d_3 -2) \geq \frac{8}{3}$, i.e. $d_3 \geq 10$.
Since also $\frac{d_3}{3} \leq \frac{d_2}{2}$, this gives $d_3 = 10$. On the other hand, if $\gamma(E) = 3$, then $d_E = 13$. For this to compute $\gamma_3(C)$ we need $\Cl_3(C) = 3$ and $\frac{1}{3}(d_3 - 2) \geq 3$, i.e. 
$d_3 \geq 11$. But then $\frac{d_3}{3} > \frac{d_2}{2}$. So this does not occur.

If $d_1 = 6$, then $C$ is either a smooth plane septic or a smooth intersection of 2 cubics in $\PP^3$. In the first case $d_2 = 7,\; d_3 = 12$. So $\frac{d_3}{3} > \frac{d_2}{2}$,
contradicting the hypothesis. In the second case, we have $d_2 = 8$. So $d_E \geq d_1 + d_2 = 14$, i.e.
$\gamma(E) \geq \frac{1}{3}(14 - 4) = \frac{10}{3} > 3.$
\end{proof}

\begin{prop} \label{prop7.9}
Suppose $\frac{d_3}{3} > \frac{d_2}{2}$. Let $E$ be a semistable bundle of rank $3$ with $h^0(E) = 5$ such that $\gamma(E) = \gamma_3(C)$. Then there exists a 
non-trivial extension \eqref{e7.4}
with $\rk F = 2, \; h^0(F) = 3$ and a line bundle $N$ with $h^0(N) = 2$. Moreover, one of the following possibilities holds
\begin{enumerate}
 \item[I.] $d_F = d_2 =8,\; d_N = d_1 =5,\; \gamma(E) = 3$.
\item[II.] $g = 14, \; d_1 = 5, \; d_2 = 7, \; d_3 = 11$ and $d_F = 7,\; d_N = 5, \; \gamma(E) = \frac{8}{3}$ or $d_F = 7, \; d_N = 6, \; \gamma(E) = 3$ 
or $d_F = 8, \; d_N = 5, \; \gamma(E) = 3$.
\item[III.] $C$ is a smooth plane septic, $d_F = d_2 = 7, \; d_N = d_1 = 6, \; \gamma(E) = 3$.   
\end{enumerate}
 
\end{prop}

\begin{proof} Note that, since $d_2 \geq 7$, the hypothesis implies that $d_3 \geq 11$.
If $E$ has no proper subbundle with $h^0 > $ rk, then Lemma \ref{lemPR} implies that 
$$
d_E \geq d_6 \geq d_3 + 3 \geq 14,
$$
giving
\begin{equation} \label{eq7.5} 
\gamma(E) \geq \frac{1}{3}(14-4) = \frac{10}{3} >3.
\end{equation}

If $E$ has a line subbundle $F$ with $h^0(F) \geq 2$, then $d_F \geq d_1$ and by semistability $d_E \geq 15$. So again $\gamma(E) > 3$.
Hence $E$ has a subbundle $F$ of rank 2 with $h^0(F) \geq 3$ and no line subbundle with $h^0 \geq 2$. If $h^0(F) \geq 4$, then $d_F \geq d_4$ by Lemma \ref{lemPR}.
So $d_E \geq \frac{3}{2}(d_2 +2) \geq \frac{27}{2}$ and $\gamma(E) > 3$.

We are left with the case $h^0(F) = 3$. So $d_F \geq d_2$ by Lemma \ref{lemPR} and, writing $N = E/F$, we have $h^0(N) \geq 2$. 
If $h^0(N) \geq 3$, then $d_N \geq d_2$ and $d_E \geq 14$. So again \eqref{eq7.5} holds.
If $h^0(N) = 2$, then 
$$
d_N \geq d_1  \quad \mbox{and} \quad d_E \geq d_1 + d_2.
$$
If $d_2 \geq 8$, then  $d_E \geq 13$ and $\gamma(E) \geq 3$ with equality if and only if $d_F = d_2 =8, \; d_N = d_1 = 5$. This gives case I.

If $d_2 = 7$, then $C$ admits a plane model of degree 7. The only cases in which $\frac{d_3}{3} > \frac{d_2}{2}$ are when $g = 14$ and the plane model has one ordinary node or cusp 
and when $C$ is a smooth plane septic. In the first case $d_3 = 11$. This gives case II.

When $C$ is a smooth plane septic, $d_1 = 6, \;d_2 = 7$ and $d_3 = 12$. So $d_E \geq 13$ with equality if and only if $d_F = d_2 = 7$ and $d_N = d_1 = 6$. This gives case III.
\end{proof}

\begin{rem}  \label{rem7.10} {\em
We can vary the hypotheses of Propositions \ref{prop7.8} and \ref{prop7.9} by replacing the condition $\gamma(E) = \gamma_3(C)$ by $ \gamma(E) < 3$. 
In Proposition \ref{prop7.9} this leaves only the possibility that $g=14, \; d_1 = 5, \; d_2 = 7$ and there exists an extension \eqref{e7.4} with $d_F = 7$ and $d_N = 5$ giving $\gamma(E) = \frac{8}{3}$.

In Proposition \ref{prop7.8}, case I disappears and in case II we can allow also $d_3 = 9$. Here there is another possibility to consider, namely that \eqref{e7.5}
gives only $\gamma(E) \geq \frac{8}{3}$. However, $d_6 \geq \min \{ \Cl(C) + 12, g+5 \}$. If $g \geq 8$, this gives $d_6 \geq 13$ and now \eqref{e7.5} 
gives $\gamma(E) \geq 3$. In the case $g = 7$, it remains possible that there exists a semistable bundle $E$ of rank 3 and degree 12 with $h^0(E) =5$ having no proper
subbundle with $h^0 > \rk$. 
}
\end{rem}

\begin{rem} {\em
We now consider the possibility that $\gamma_3(C) < \Cl_3(C)$. 

If $\frac{d_3}{3} \leq \frac{d_2}{2}$, then by \eqref{eq6.1} this happens if and only if either
$$
\Cl_3(C) = 3, \; d_3 \leq 10 \quad \mbox{or} \quad \Cl_3(C) = \frac{8}{3}, \; d_3 = 9.
$$
If this happens, then $\gamma_3(C)$ is computed by a bundle $E$ with $d_E = d_3$ and $h^0(E) = 4$ and by no bundle with $h^0(E) = 5$, except possibly in case II of Proposition 
\ref{prop7.8} when $\Cl_3(C) = 3$.

In particular, $\gamma_3(C) < \Cl_3(C)$ and is computed by a bundle with $h^0 = 4$ in the following cases
\begin{itemize}
 \item $g= 7$ or 8, $\gamma_3(C) = \frac{7}{3}$;
\item $g = 13$, if $d_1=5,\; d_2 = 7, \; \gamma_3(C) = \frac{8}{3}$ (in this case $\Cl_3(C) = 3$ by Theorem \ref{thm5.5}  and $d_3 = 10$);
\item $9 \leq g \leq 12$ if $d_1 =5, \; d_2 = 7$ and either $\Cl_3(C) = 3$ or $d_3 = 9, \; \gamma_3(C) = \frac{7}{3}$ or $\frac{8}{3}$ (in these cases we know only that
$\frac{8}{3} \leq \Cl_3(C) \leq 3$ and $ 9 \leq d_3 \leq 10$);
\item $g=9,\; d_1 = 5, \; d_2 = 8, \gamma_3(C) = \frac{8}{3}$ (such curves exist and $\Cl_3(C)=3$ by \cite[Theorem 6.8]{ln2}; moreover $d_3=10$);
\item smooth intersections of two cubics in $\PP^3$, $\gamma_3(C) = \frac{7}{3}$.  
\end{itemize}

If $\frac{d_3}{3} > \frac{d_2}{2}$, there are no semistable bundles $E$ of rank 3 with $h^0(E) = 4$ and $\gamma(E) < 3$. The only possibility for a bundle computing 
$\gamma_3(C) < \Cl_3(C)$ is when $C$ is the normalization of a plane curve of degree 7 with one ordinary node or cusp and $h^0(E) = 5$ (see Proposition \ref{prop7.9} case II).
We do not know whether such a bundle exists.
} 
\end{rem}

\section{bundles of degree 12 with 5 sections}

Let $C$ be a curve of Clifford index 3. It is an interesting question to determine whether there exists a semistable bundle $E$ of rank 3 with $d_E = 12$, $h^0(E) = 5$ 
and hence $\gamma(E)=\frac83$. In this section we shall prove that $E$ exists when $g=7$. In doing so, we shall show that the existence of $E$ is equivalent to the 
existence of a semistable bundle $F$ of rank $2$ with $d_F=12$, $h^0(F)=5$ and hence $\gamma(F)=3$. 

Now suppose $g \geq 7$ and let $E$ be a semistable bundle of rank 3 with $d_E = 12, \; h^0(E) = 5$.

\begin{lem} \label{lem8.1}
 $E$ is stable.
\end{lem}

\begin{proof}
If $E$ is not stable, then it possesses either a line subbundle $L$ of degree 4 or a stable rank-2 subbundle $F$ of degree 8.
In the first case, $h^0(L) \leq 1$ and $E/L$ is semistable of degree 8, so $h^0(E/L) \leq 3$ and $h^0(E)\le4$. In the second case, $h^0(F) \leq 3$ and $h^0(E/F) \leq 1$ and again $h^0(E)\le4$.    
\end{proof}

\begin{lem}  \label{lem8.2}
$E$ is generated. 
\end{lem}

\begin{proof}
It is easy to see that that a proper subbundle of $E$ cannot have $h^0 = 5$. So, if $E$ is not generated, 
there exists a stable bundle $E'$ of rank 3 with $d_{E'} = 11$ and $h^0(E') = 5$. This contradicts Propositions \ref{prop7.8} and \ref{prop7.9}.
\end{proof}

Recall the sequence \eqref{eq2.6}.

\begin{lem} \label{lem8.3}
$D(E)$ is stable of rank $2$ and degree $12$ with $h^0(D(E)) = 5$. 
\end{lem}

\begin{proof}
It is clear that $D(E)$ has rank 2 and degree 12 and $h^0(E) \geq 5$. If $D(E)$ is not stable, it has a quotient line bundle $G$ with $d_G \leq 6$. So $h^0(G) \leq 2$. 
Consider the diagram (3.6). Since $G$ is generated, we must have $h^0(G) = \dim W = 2$. So $D_{W,G}$ is a line bundle of degree $d_G$. Since $h^0(E^*) = 0$, the map $\alpha$ 
is non-zero implying that $\coker (\alpha)$ and $\coker (\beta)$ are torsion sheaves. So $D_{W,G}$ generates a line subbundle $L$ of $E$ with $d_L \geq d_G \geq 5$, contradicting
the stability of $E$.

If $h^0(D(E)) > 5$, it would contradict the fact that $\Cl_2(C) = 3$.
\end{proof}

\begin{lem}  \label{lem8.4}
Let $F$ be a stable bundle of rank $2$ and degree $12$ with $h^0(F) = 5$. Then $F$ is generated and $D(F)$ is a stable rank-$3$ bundle of degree $12$ with $h^0(D(F)) = 5$. 
\end{lem}

\begin{proof}
If $F$ is not generated, there exists a stable bundle of rank 2 and degree 11 with $h^0 = 5$, contradicting the fact that $\Cl_2(C) = 3$.

Certainly $D(F)$ is generated and $h^0(D(F)^*) = 0$. Any quotient line bundle of $D(F)$ has $h^0 \geq 2$ and hence degree $\geq 5$. So, if $D(F)$ is not stable, 
it possesses a generated stable rank-2 quotient bundle $G$ of degree $\leq 8$. We use diagram (3.6) with $E$ replaced by $F$. Since $\Cl_2(C) = 3$, $h^0(G) \leq 3$ and hence $=3$, and $\dim W =3$.
Moreover, $d_G \geq d_2 \geq 7$. Since $\alpha \neq 0$, $D_{W,G}$ generates a line subbundle of $F$ of degree $d_F\geq d_G \geq 7$, contradicting the stability of $F$. 
\end{proof}

Combining the above lemmas, we immediately obtain the following theorem.

\begin{theorem} \label{thm8.5}
Let $C$ be a curve of Clifford index $3$. The dual span construction gives a bijection between the set of stable rank-$3$ bundles of degree $12$ with $h^0 = 5$ and the set of stable 
rank-$2$ bundles of degree $12$ with $h^0 = 5$. 
\end{theorem}

\begin{theorem} \label{thm8.6}
Let $C$ be a curve of genus $7$ with $\Cl(C) = 3$. Then
\begin{itemize}
 \item[(i)] $d_1 = 5, \; d_2 = 7,\; d_3 = 9$;
\item[(ii)] $\Cl_3(C) = 3$;
\item[(iii)] the smallest degree $d_E$ for which there exists a semistable bundle $E$ of rank $3$ with $h^0(E) = 4$ is $d_E=9$ and any such bundle 
is stable and generated and of the form
$E \simeq E_L$, where $E_L$ is given by the exact sequence
$$
0 \ra E_L^* \ra H^0(L) \otimes \cO_C \ra L \ra 0
$$
with $L$ a line bundle of degree $9$ with $h^0(L) = 4$;
\item[(iv)] the smallest degree $d_E$ for which there exists a semistable bundle $E$ of rank $3$ with $h^0(E) = 5$ is $d_E=12$. 
\end{itemize}
\end{theorem}

\begin{proof}
(i) is obvious and (ii) is Theorem \ref{thm4.6}.

(iii): The fact that $d_E = d_3 = 9$ is included in Proposition \ref{prop6.3}. If $E$ is not stable, then it possesses 
either a line subbundle $M$ of degree 3 with $E/M$ semistable or a semistable rank-2 subbundle $F$ of degree 6. In either case it is easy to see that $h^0(E) \leq 3$. 

To see that $E$ is generated, note first that it is easy to see that no proper subbundle can have $h^0 = 4$. So if $E$ is not generated, there exists a stable bundle of rank $3$ and degree 8 with $h^0 = 4$, contradicting the fact that $d_E=9$ is the smallest degree for which such a bundle exists.

So we can apply the dual span construction to $E$. Then $D(E)$ is a line bundle $L$ of degree 9 with $h^0(L) \geq 4$ and hence $= 4$. It follows at once that $E \simeq E_L$.

(iv): By Remark \ref{rem7.10}, we have $\gamma(E) \geq \frac{8}{3}$. So $d_E \geq 12$. According to Theorem \ref{thm8.5}, it is sufficient to find a stable rank-2 bundle of degree
12 with $h^0 = 5$. For examples of this, see \cite{bf} or \cite[Proposition 7.7]{ln3}.
\end{proof}

\begin{rem}  \label{rem8.7} {\em
If $C$ has genus 8 with $\Cl(C) =3$, then 
$$
d_1 = 5, \quad d_2 = 7 \; \mbox{or} \; 8 \quad \mbox{and} \quad d_3 =9.
$$
Parts (ii) and (iii) of Theorem \ref{thm8.6} remain true. If $E$ is a semistable bundle of rank 3 with
$h^0(E) =5$ and $d_2 =7$, then $d_E \geq 12$, but we do not know whether such a bundle with 
$d_E = 12$ exists. If $d_2 = 8$, then $d_E \geq 13$. In any case there are no bundles with $h^0 =5$ which compute $\gamma_3(C)$. 
}
\end{rem}

\begin{rem} {\em
If $C$ has genus 9 with $\Cl(C) =3$, then $d_1 =5$, and using Serre duality one sees that either 
$d_2 = 7,\; d_3 = 9$ or $d_2 = 8, \; d_3 = 10$. 
It follows from Proposition \ref{prop7.8} that there does not exist a bundle with $h^0 =5$ computing $\gamma_3(C)$.
}
\end{rem}

\section{Comments and questions}

{\bf Comment 9.1.} Let $C$ be the normalization of an irreducible plane curve of degree 7 with 
$\Cl(C) = 3$. Let $H$ denote the hyperplane bundle on $C$ and write $E_H := D(H)$. Consider extensions of the form
\begin{equation}  \label{eq9.1}
0 \ra E_H \ra E \ra H \ra 0
\end{equation}
defining bundles $E$ of rank 3 with $d_E = 14$.
If all sections of $H$ lift to $E$, then $h^0(E) = 6$. So, by Proposition \ref{prop3.1}, $E$ cannot be semistable.
However, according to \cite[Proposition 2.6]{ln2}, in any non-trivial extension \eqref{eq9.1} for which $h^0(E) = 6$, $E$ is semistable. Moreover, 
such an extension exists if and only if $h^0(E_H \otimes E_H) \geq 10$. It follows that $h^0(E_H\otimes E_H)\le9$ (in fact $h^0(E_H\otimes E_H)=9$). This gives a negative
answer to \cite[Question 8.4]{ln2} in this case.\\

{\bf Comment 9.2.} We have substantially reduced the list of curves of Clifford index 3 for which it is 
possible that $\Cl_3(C) = \frac{8}{3}$. In fact, any such curve must have genus $g$ with $9 \leq g \leq 12$ and is representable by a singular plane septic (compare \cite[Question 8.7]{ln2}).\\

{\bf Question 9.3.} On a curve of genus 8 with Clifford index 3, what is the minimal degree $d_E$ of a semistable bundle $E$ of rank 3 with $h^0(E) = 5$?
(At the moment we know that $d_E \geq 12$ and if $d_2 = 8$, then $d_E \geq 13$ (see Remark \ref{rem8.7}).)\\

{\bf Question 9.4.} Let $C$ be a curve of genus 14 with Clifford index 3.
 Does there exist a semistable $E$ with $d_E = 12,\; h^0(E) = 5$? (see Proposition \ref{prop7.9} II.)

\end{document}